\documentclass{amsart}

\usepackage{amssymb,amsmath}

\newtheorem{theorem}{Theorem}[section]
\newtheorem{lemma}[theorem]{Lemma}

\newtheorem{proposition}[theorem]{Proposition}
\newtheorem{corollary}[theorem]{Corollary}
\newtheorem{remark}[theorem]{Remark}
\newtheorem{question}[theorem]{Question}

\newtheorem{claim}{Claim}

\title[On the equation of string oscillation]
{Separately twice differentiable functions and the equation of string oscillation}

\author[T.~Banakh, V.~Mykhaylyuk]{Taras Banakh and Volodymyr Mykhaylyuk}
\address[T.Banakh]{Ivan Franko National University of Lviv, Ukraine and\newline Uniwersytet Humanistyczno-Przyrodniczy Jana Kochanowskiego, Kielce, Poland.}
\email{t.o.banakh@gmail.com}
\urladdr{http://www.franko.lviv.ua/faculty/mechmat/Departments/Topology/bancv.html}
\address[V.Mykhaylyuk]{Department of Mathematical Analysis,
Yuriy Fedkovych Chernivtsi National University,
 Kotsjubynskogo str. 2, Chernivtsi 58012, Ukraine.}
\email{vmykhaylyuk@ukr.net}

\begin{document}

\begin{abstract}
We prove that for every separately twice differentiable function $f:\mathbb R^2\to\mathbb
R$ with that $f''_{xx}=f''_{yy}$ there exist twice differentiable
functions $\varphi, \psi:\mathbb R\to\mathbb R$ such that $f(x,y)=\varphi(x+y) + \psi(x-y)$.
\end{abstract}
\maketitle

\section{Introduction}

Let $f:X\times Y\to Z$ be a mapping defined on a product $X\times Y$ and valued in $Z$. For any
$x\in X$ and $y\in Y$ we define mappings $f^x:Y\to Z$ and
$f_y:X\to Z$ by the equalities: $f^x(y)=f_y(x)=f(x,y)$.
We say that {\it a mapping $f$ separately has $P$} for some
property $P$ of mappings (continuity, differentiability, twice differentiability, etc.) if
for any $x\in X$ and $y\in Y$ the mappings $f^x$ and $f_y$ have
$P$.

R.~Baire in the fifth section of his PhD thesis \cite{Baire} raised a problem
of solving differential equations with partial derivatives
under minimal requirements, that is, a problem of solving a
differential equation in the class of functions satisfied strictly
necessary conditions for the existence of expressions which are
contained in this equation. He proved that a jointly
continuous separately differentiable function $f:\mathbb
R^2\to\mathbb R$ is a solution of the equation
$$
\frac{\partial f}{\partial x} + \frac{\partial f}{\partial y}=0
\eqno(1)
$$
if and only if there exists
a differentiable function $\varphi:\mathbb R\to\mathbb R$ such
that $f(x,y)=\varphi(x-y)$ for any $x,y\in \mathbb R$.
In connection with this R.~Baire naturally raised the following question.

\begin{question} \label{q1}\cite[p.118]{Baire}
Let $f:\mathbb
R^2\to\mathbb R$ be a separately differentiable solution of (1).
Does there exist a differentiable function $\varphi:\mathbb
R\to\mathbb R$ such that $f(x,y)=\varphi(x-y)$ for any $x,y\in
\mathbb R$?
\end{question}

Note that a similar result to Baire's one was independently
obtained in \cite{Cher}, where Question \ref{q1} was formulated too.

During last 100 years differential equations with partial derivatives were studied intensively by many mathematicians.
Their investigations lead to the appearance of the theory  of partial differential equations, the basic notion of which is the notion of generalized function. Note that the generalized functions do not give an immediate possibility to solve differential equations with partial derivatives in the classes of separately differentiable functions or separately
twice differentiable functions, which is tightly connected with the fact that the partial derivative of a separately differentiable function can be locally non-integrable on  a set of a positive measure. Therefore, an application of the method of generalized functions to solving differential equations with partial derivatives in the classes of separately differentiable functions leads to the following question.
Is it true that  a generalized function generated by a bounded separately continuous solution of some equation is also a solution of the same equation?

The following argument shows that this question in general case has a negative answer. It is well known that $u''_{xy}=u''_{yx}$ for every generalized function $u$. It is easy to construct an example of  a function $f:\mathbb R^2\to \mathbb R$ which everywhere has mixed derivatives which are different at  some point (i.e.,
$f(x,y)=\frac{xy(x^2-y^2)}{x^2+y^2}$, if $x^2+y^2>0$, and
$f(0,0)=0$). Since two almost everywhere equal functions on $\mathbb R^2$ generate the same
generalized function, it is natural to assume that
$f''_{xy}\stackrel{\mbox{\footnotesize{a.e.}}}{=}f''_{yx}$ for every function $f:\mathbb R^2\to \mathbb R$ which everywhere has mixed derivatives of the second order. In particular, the equality $f''_{xy}\stackrel{\mbox{\footnotesize{a.e.}}}{=}f''_{yx}$ was obtained in \cite{Tol1} for every function $f:\mathbb R^2\to \mathbb R$ which everywhere has all partial derivatives of the second order. But for an arbitrary function $f:\mathbb R^2\to \mathbb R$ this assertion is not true. G.Tolstov in \cite{Tol2} constructed two functions $F,G:[0,1]^2\to \mathbb R$ which satisfy the following conditions:

$a)$ $F$ has everywhere mixed derivatives
$F''_{xy}$ and $F''_{yx}$, but $F''_{xy}- F''_{yx}=\chi_E$, where
$E$ is a set of a positive measure and $\chi_E$ is the characteristic  function of $E$;

$b)$ $G$ has almost everywhere mixed derivatives $G''_{xy}$ and $G''_{yx}$, but
$G''_{xy}\stackrel{\mbox{\footnotesize{a.e.}}}{\ne}G''_{yx}$ on
$[0,1]^2$.

Thus $F$ and $G$ generate generalized solutions of the equation $u''_{xy}=u''_{yx}$ but fail to be classical solutions of this equation. On the other hand, $F$ is a bounded classical solution of the equation $u''_{xy}-
u''_{yx}=\chi_E$ which has no generalized solutions.

In \cite{M, BPPT, KM} properties of solutions of the equation
$$
\frac{\partial u}{\partial x}\cdot \frac{\partial u}{\partial
y}=0, \eqno(2)
$$
were studied. It was obtained that separately continuous solution of
(2) depends at most on one variable.

A technique from \cite{Baire} was developed in \cite{MM}. It was established that Question \ref{q1} has a positive answer, which implies the following result.

\begin{theorem}\label{t1}\cite[Theorem 6.2]{MM}
Let a function
$f:\mathbb R^2\to\mathbb R$ has all partial derivatives of the second order and
$$
f''_{xx}(p)=f''_{yy}(p)\,\,\,\,\,\mbox{and}\,\,\,\,\,f''_{xy}(p)=f''_{yx}(p)
$$
for every $p\in \mathbb R^2$. Then there exist twice differentiable functions $\varphi, \psi:\mathbb
R\to\mathbb R$ such that
$$
f(x,y)=\varphi(x+y) + \psi(x-y)
$$
for every $x,y\in\mathbb R$.
\end{theorem}

In connection with this the following question on solutions of the equation of string oscillation
$$
\frac{\partial^2 u}{\partial x^2} = \frac{\partial^2 u}{\partial
y^2}, \eqno(3)
$$
arises naturally.

\begin{question}\label{q2}\cite[Question 6.3]{MM}
Let a function $f:\mathbb R^2\to\mathbb
R$ has partial derivatives $f''_{xx}$ and $f''_{yy}$ and
$$
f''_{xx}(p)=f''_{yy}(p)
$$
for every $p\in \mathbb R^2$. Do there exist twice differentiable
functions $\varphi, \psi:\mathbb R\to\mathbb R$ such that
$$
f(x,y)=\varphi(x+y) + \psi(x-y)
$$
for every $x,y\in\mathbb R$?
\end{question}

In this paper we develop methods from \cite{Baire, MM}, and introduce and investigate auxiliary functions $\lambda_1^{(2)}$ and $\lambda_2^{(2)}$, generated by a twice differentiable function. Using these functions we establish some properties of separately twice differentiable functions and in Theorem~\ref{t7} shall give a positive answer to Question \ref{q2}.

\section{Auxiliary functions  $\lambda_2^{(2)}$ and $\lambda_1^{(2)}$}

 Firstly we introduce an auxiliary function $\lambda_2^{(2)}$ which is generated by the second order divided difference, and investigate its properties.

Let $f:\mathbb R\to \mathbb R$ be a function. Denote
$$
r_f^{(2)}(p,s)=r_f^{(2)}(x_1,x_2,s)=\frac{f(x_1)-f(x_1+s)-f(x_2-s)+f(x_2)}{s(x_2-x_1-s)}
$$
for every $p=(x_1,x_2)\in\mathbb R$, $x_1<x_2$, and $s\in (0,x_2-x_1)$.

It is easy to see that
$r_f^{(2)}(x_1,x_2,s)=r_f^{(2)}(x_1,x_2,x_2-x_1-s)$. Thus it is sufficient to study the properties of the function $r_f^{(2)}(x_1,x_2,s)$ for $s\in (0,\frac{x_2-x_1}{2}]$.

Note that $r_f^{(2)}(x,y,s)=2a$ for any function $f(x)=ax^2+bx+c$ and every admissible $x,y,s\in \mathbb R$. Besides, $r_{u+v}^{(2)}=r_u^{(2)}+
r_v^{(2)}$ for every functions $u,v:\mathbb R\to\mathbb R$.

For every $\varepsilon>0$ and $x\in\mathbb R$ denote by
$\Delta_2^{(2)}(\varepsilon, f, x)$ the set of all
$\delta\in(0,1]$ such that
$$
|r_f^{(2)}(p',s')-r_f^{(2)}(p'',s'')|\leq\varepsilon
$$
for every $p'=(x_1',x_2'),
p''=(x_1'',x_2'')\in (x-\delta, x)\times(x,x+\delta)$ and
$s'\in(0,x_2'-x_1')$, $s''\in(0,x_2''-x_1'')$.

The function $\lambda_2^{(2)}(\varepsilon,f):\mathbb R\to \mathbb R$
is defined by

$$
\lambda_2^{(2)}(\varepsilon,f)(x) = \left \{\begin{array}{rr}
 {\rm sup}\, \Delta_2^{(2)}(\varepsilon, f, x),
&
 {\rm if}\quad \Delta_2^{(2)}(\varepsilon, f, x)\ne \O;
\\
  0,
&
 {\rm if}\quad \Delta_2^{(2)}(\varepsilon, f, x)=\O.
  \end{array} \right .
$$

\begin{proposition}\label{p1}
Let $f:\mathbb R\to\mathbb R$ be a function twice differentiable at a point $x_0\in \mathbb R$ and
$\varepsilon>0$. Then $\lambda_2^{(2)}(\varepsilon,f)(x_0)>0$.
\end{proposition}

\begin{proof} Note that it is sufficient to consider the case when $x_0=0$.

Consider the function $g(x)=f(x)-f(0)-f'(0)x-\frac{1}{2}f''(0)x^2$.
Since $g'(0)=g''(0)=0$, there exists $\delta\in(0,1]$ such that
$|\frac{g'(\tau)}{\tau}|<\frac{\varepsilon}{8}$ for every
$0<|\tau|<\delta$. Take any $x_1\in (-\delta,0)$, $x_2\in
(0,\delta)$ and $s\in (0,\frac{x_2-x_1}{2}]$. Note that
$t=(x_2-x_1)-s\geq \frac{x_2-x_1}{2}$. Taking into account that $x_1<0<x_2$, we obtain that
$|x|\leq x_2-x_1$ for every $x\in (x_1,x_2)$. Therefore $|\frac{x}{t}|\leq 2$. Now we have
$$
\begin{aligned}
|r_g^{(2)}(x_1,x_2,s)|&=\left|\frac{g(x_1)-g(x_1+s)-g(x_2-s)+g(x_2)}{s(x_2-x_1-s)}\right|=\\
&=\left|\frac{g'(x_1+\theta_1s)-g'(x_2-\theta_2s)}{t}\right|\leq\\
&\leq\left|\frac{x_1+\theta_1s}{t}\right|\cdot\left|\frac{g'(x_1+\theta_1s)}{x_1+\theta_1s}\right|+
\left|\frac{x_2-\theta_2s}{t}\right|\cdot\left|\frac{g'(x_2-\theta_2s)}{x_2-\theta_2s}\right|\leq\\
&\leq 2\frac{\varepsilon}{8}+
2\frac{\varepsilon}{8}=\frac{\varepsilon}{2},
\end{aligned}
$$
where $\theta_1, \theta_2 \in (0,1)$ are chosen by Lagrange Theorem. Therefore for every $p'=(x_1',x_2'),
p''=(x_1'',x_2'')\in (-\delta, 0)\times(0,\delta)$ and
$s'\in(0,x_2'-x_1')$, $s''\in(0,x_2''-x_1'')$ we have

$$
|r_f^{(2)}(p',s')-r_f^{(2)}(p'',s'')|=
|r_g^{(2)}(p',s')-r_g^{(2)}(p'',s'')|\leq$$

$$|r_g^{(2)}(p',s')|+|r_g^{(2)}(p'',s'')|\leq\frac{\varepsilon}{2}+
\frac{\varepsilon}{2}= \varepsilon.$$

Thus $\delta\in\Delta_2^{(2)}(\varepsilon, f, 0)$ and
$\lambda^{(2)}(\varepsilon,f)(0)\geq\delta>0$.
\end{proof}

\begin{remark}\label{r1}
Note that
$\lim\limits_{s\to+0}r_f^{(2)}(x_1,x_2,s)=\frac{f'(x_2)-f'(x_1)}{x_2-x_1}$ for every function $f$ differentiable at points $x_1$ and $x_2$. For such function $f$ put $r_f^{(2)}(x_1,x_2,0)=r_f^{(2)}(x_1,x_2,x_2-x_1)=\frac{f'(x_2)-f'(x_1)}{x_2-x_1}$.
Besides, $f''(x_0)=\mathop{\lim\limits_{x_1\to x_0-0}}\limits_{x_2\to x_0+0}r_f^{(2)}(x_1,x_2,s)$ for a  function $f$ twice differentiable at point $x_0$. Therefore it is natural to put $r_f^{(2)}(x_0,x_0,0)=f''(x_0)$. This implies that for a function $f$ which is differentiable on $\mathbb R$ and twice differentiable at $x_0\in\mathbb R$ in the definition of the set $\Delta_2^{(2)}(\varepsilon, f, x_0)$ we can put
$p'=(x_1',x_2'),p''=(x_1'',x_2'')\in(x-\delta, x]\times[x,x+\delta)$, $s'\in[0,x_2'-x_1']$ and  $s''\in[0,x_2''-x_1'']$.
\end{remark}

The following proposition gives joint properties of $\lambda_2^{(2)}$.

\begin{proposition}\label{p2}
Let $Y$ be a topological space, $f:\mathbb R\times Y\to\mathbb R$ be a function continuous with respect to $y$ and
$\varepsilon>0$. Then the function $g:\mathbb R\times Y\to\mathbb R$, $g(x,y)=\lambda_2^{(2)}(\varepsilon,f_y)(x)$, is jointly upper semicontinuous.
\end{proposition}

\begin{proof} Let $x_0\in \mathbb R$, $y_0\in Y$,
$\gamma=g(x_0,y_0)$ and $\eta >0$. If $\gamma + \eta > 1$, then
$g(x,y)\leq 1<\gamma + \eta$ for every $(x,y)\in \mathbb
R\times Y$.

Consider the case when $\gamma + \eta \leq 1$. Then $\delta_0 =
\gamma + \frac{\eta}{3}\leq 1$. Since ${\rm sup}\,
\Delta_2^{(2)}(\varepsilon, f_{y_0},
x_0)=g(x_0,y_0)=\gamma<\delta_0$, $\delta_0\not\in
\Delta_2^{(2)}(\varepsilon, f_{y_0}, x_0)$, i.e. there exist
$p'=(x_1',x_2'), p''=(x_1'',x_2'')\in (x_0-\delta_0,
x_0)\times(x_0,x_0+\delta_0)$ and $s'\in(0,x_2'-x_1')$,
$s''\in(0,x_2''-x_1'')$ such that
$$
|r_{f_{y_0}}^{(2)}(p',s')-r_{f_{y_0}}^{(2)}(p'',s'')|>\varepsilon.
$$
Since $f$ is continuous with respect to $y$, the function $\varphi(y)=r_{f_{y}}^{(2)}(p',s')-
r_{f_{y}}^{(2)}(p'',s'')$ is continuous. Therefore there exists a neighborhood $V$ of
$y_0$ in $Y$ such that
$$
|r_{f_{y}}^{(2)}(p',s')-r_{f_{y}}^{(2)}(p'',s'')|>\varepsilon
$$
for every $y\in V$. Put
$\rho=\min\{x_0-x'_1,x_0-x''_1,x'_2-x_0,x''_2-x_0,\frac{\eta}{3}\}>0$,
$U=(x_0-\rho,x_0+\rho)$ and $\delta_1=\gamma+\frac{2\eta}{3}$. Now we have
$x'_1,x''_1\in (x-\delta_1,x)$ and $x'_2,x''_2\in (x,x+\delta_1)$
for every $x\in U$. Therefore $\delta_1\not\in
\Delta_2^{(2)}(\varepsilon, f_{y}, x)$ for every $x\in U$ and
$y\in V$. Thus $\Delta_2^{(2)}(\varepsilon, f_{y}, x)\subseteq
(0,\delta_1)$ and $g(x,y)\leq \delta_1 < \gamma+\eta$ for every
$x\in U$ and $y\in V$.

Hence $g$ is upper semicontinuous at $(x_0,y_0)$.\hfill$\Box$

Now we introduce a function $\lambda_1^{(2)}$ which gives for a twice differentiable function a possibility to
estimate the exactness of approximation of derivative by the first order divided difference.

Let $f:\mathbb R\to \mathbb R$ be a function. Denote
$$
r_f(p)=r_f(x_1,x_2)=\frac{f(x_1)-f(x_2)}{x_2-x_1}
$$
for every $p=(x_1,x_2)\in\mathbb R^2$, $x_1\ne x_2$.
For every $\varepsilon>0$ and $x\in\mathbb R$ denote by
$\Delta^{(2)}_1(\varepsilon, f, x)$ the set of all
$\delta\in(0,1]$ such that
$$
|r_f(p')-r_f(p'')|\leq\varepsilon \cdot{\rm
max}\{|x_2'-x_1'|,|x_2''-x_1''|\}
$$
for every $p'=(x_1',x_2'),
p''=(x_1'',x_2'')\in (x-\delta, x)\times(x,x+\delta)$ with
$x_1'+x_2'=x_1''+x_2''$.

The function $\lambda_2^{(1)}(\varepsilon,f):\mathbb R\to \mathbb R$
is defined by

$$
\lambda_1^{(2)}(\varepsilon,f)(x) = \left \{\begin{array}{rr}
 {\rm sup}\, \Delta_1^{(2)}(\varepsilon, f, x),
&
 {\rm if}\quad \Delta_1^{(2)}(\varepsilon, f, x)\ne \O;
\\
  0,
&
 {\rm if}\quad \Delta_1^{(2)}(\varepsilon, f, x)=\O.
  \end{array} \right .
$$

\begin{proposition}\label{p3}
Let $f:\mathbb R\to\mathbb R$ be a twice differentiable at $x_0\in \mathbb R$ function and
$\varepsilon>0$. Then $\lambda^{(2)}_1(\varepsilon,f)(x_0)>0$.
\end{proposition}

\begin{proof} Consider the function
$g(x)=f(x)-f(x_0)-f'(x_0)(x-x_0)-\frac{1}{2}f''(x_0)(x-x_0)^2$.
Since $g'(x_0)=g''(x_0)=0$, $g(x)=o((x-x_0)^2)$. Hence there exists
$\delta\in(0,1]$ such that $|g(x_0+t)|<\frac{\varepsilon}{2}\,t^2$
for every $0<|t|<\delta$. Note that
$r_f(x_1,x_2)=f'(x_0)+\frac{1}{2}f''(x_0)(x_2+x_1-2x_0)+r_g(x_1,x_2)$
for arbitrary distinct $x_1,x_2\in\mathbb R$. Now for every
$p'=(x_1',x_2'), p''=(x_1'',x_2'')\in (x_0-\delta,
x_0)\times(x_0,x_0+\delta)$ with $x_1'+x_2'=x_1''+x_2''$ we have
$$
|r_f(p')-r_f(p'')|= |r_g(p')-r_g(p'')|\leq
\frac{|g(x_2')|}{x_2'-x_0} +\frac{|g(x_1'|}{x_0-x_1'}+
\frac{|g(x_2'')|}{x_2''-x_0}+\frac{|g(x_1'|}{x_0-x_1'}
\leq
$$
$$
\frac{\varepsilon}{2}
((x_2'-x_0)+(x_0-x_1')+(x_2'-x_0)+(x_0-x_1'))\leq  \varepsilon
\cdot{\rm max}\{|x_2'-x_1'|,|x_2''-x_1''|\} .
$$
Thus $\delta\in\Delta^{(2)}_1(\varepsilon, f, x_0)$ and
$\lambda^{(2)}(\varepsilon,f)(x_0)\geq\delta>0$.\end{proof}

\begin{proposition}\label{p4}
Let $f:\mathbb R\to\mathbb R$ be a differentiable at $x_0\in \mathbb R$ function,
$\varepsilon>0$ and $0<|s|<\lambda^{(2)}_1(\varepsilon,f)(x_0)$.
Then $|\frac{f(x_0+s)-f(x_0-s)}{2s}-f'(x_0)|\leq 2\,\varepsilon
s$.
\end{proposition}

\begin{proof} Note that it is sufficient to prove this proposition for $s>0$. Let $0<t<s$. Since there exists $\delta\in
\Delta^{(2)}_1(\varepsilon, f, x_0)$ with $\delta>s$,
$$
|r_f(x_0-s,x_0+s)-r_f(x_0-t,x_0+t)|\leq2\,\varepsilon s.
$$
It remains to take into account that $\lim\limits_{t\to 0}r_f(x_0-t,x_0+t)=f'(x_0)$.\end{proof}

The following proposition can be proved by analogy with Proposition \ref{p2}.

\begin{proposition}\label{p5}
Let $Y$ be a topological space,
$f:\mathbb R\times Y\to\mathbb R$ be a continuous with respect to $y$ function and $\varepsilon>0$. Then the function $g:\mathbb R\times
Y\to\mathbb R$, $g(x,y)=\lambda_1^{(2)}(\varepsilon,f_y)(x)$,
is jointly upper semicontinuous.
\end{proposition}

\begin{proof} Let $x_0\in \mathbb R$, $y_0\in Y$,
$\gamma=g(x_0,y_0)$ and $\eta >0$. If $\gamma + \eta > 1$, then
$g(x,y)\leq 1<\gamma + \eta$ for every $(x,y)\in \mathbb
R\times Y$.

Consider the case when $\gamma + \eta \leq 1$. Òîä³ $\delta_0 =
\gamma + \frac{\eta}{3}\leq 1$. Since ${\rm sup}\,
\Delta_1^{(2)}(\varepsilon, f_{y_0},
x_0)=g(x_0,y_0)=\gamma<\delta_0$, $\delta_0\not\in
\Delta_1^{(2)}(\varepsilon, f_{y_0}, x_0)$, i.e. there exist
$p'=(x_1',x_2'), p''=(x_1'',x_2'')\in (x_0-\delta_0,
x_0)\times(x_0,x_0+\delta_0)$ such that $x_1'+x_2'=x_1''+x_2''$ and
$$
|r_{f_{y_0}}(p')-r_{f_{y_0}}(p'')|>\varepsilon\cdot{\rm
max}\{|x_2'-x_1'|,|x_2''-x_1''|\} .
$$
Since $f$ is continuous with respect to $y$, the function $\varphi(y)=r_{f_{y}}(p')- r_{f_{y}}(p'')$ is continuous.
Therefore there exists a neighborhood $V$ of $y_0$ in $Y$ such that
$$
|r_{f_{y}}(p')-r_{f_{y}}(p'')|>\varepsilon\cdot{\rm
max}\{|x_2'-x_1'|,|x_2''-x_1''|\}
$$
for every $y\in V$. Put
$\rho=\min\{x_0-x'_1,x_0-x''_1,x'_2-x_0,x''_2-x_0,\frac{\eta}{3}\}>0$,
$U=(x_0-\rho,x_0+\rho)$ ³ $\delta_1=\gamma+\frac{2\eta}{3}$. Then
$x'_1,x''_1\in (x-\delta_1,x)$ and $x'_2,x''_2\in (x,x+\delta_1)$
for every $x\in U$. Thus $\delta_1\not\in
\Delta_1^{(2)}(\varepsilon, f_{y}, x)$ for every $x\in U$ and
$y\in V$. Hence $\Delta_1^{(2)}(\varepsilon, f_{y}, x)\subseteq
(0,\delta_1)$ and $g(x,y)\leq \delta_1 < \gamma+\eta$ for every
$x\in U$ and $y\in V$.

Thus $g$ is jointly upper semicontinuous at $(x_0,y_0)$.\end{proof}

\section{Separately twice differentiable pointwise
changeable functions}

In this section we give an answer to Question \ref{q2} in the case of continuous partial derivatives $f'_x$ and $f'_y$.

\begin{proposition}\label{p6}
Let $f:\mathbb R\to\mathbb R$ be a twice differentiable at $x_0$ function, $\varepsilon>0$  and
$\delta={\rm min}\{ \lambda^{(2)}_1(\varepsilon, f)(x_0),
\lambda^{(2)}_2(\varepsilon, f)(x_0)\}$. Then
$$
|f'(x_0)-\frac{f(x)-f(x_0)}{x-x_0}+\frac{1}{2}f''(x_0)(x-x_0)|\leq
\frac{5}{2}\varepsilon |x-x_0|
$$
for every
$x\in (x_0-\delta, x_0)\cup(x_0,x_0+\delta)$.
\end{proposition}

\begin{proof} Note that it is sufficient to consider the case of $x\in
(x_0,x_0+\delta)$.

Put $t=x-x_0$. Since $0<t\leq
\lambda^{(2)}_2(\varepsilon, f)(x_0)$, taking into account Remark \ref{r1} we have
$$
|f''(x_0)-r^{(2)}_f(x_0-t,x_0+t,t)|\leq\varepsilon.
$$
Besides, note that
$$
r^{(2)}_f(x_0-t,x_0+t,t)=2\frac{f(x_0+t)-f(x_0)}{t^2} -
\frac{f(x_0+t)-f(x_0-t)}{t^2}.
$$
Thus,
$$
|f''(x_0)-2\frac{f(x_0+t)-f(x_0)}{t^2} +
\frac{f(x_0+t)-f(x_0-t)}{t^2}|\leq\varepsilon.
$$

On the other hand, the inequality $0<t\leq \lambda^{(2)}_1(\varepsilon,
f)(x_0)$ and Proposition \ref{p4} imply
$$
|\frac{f(x_0+t)-f(x_0-t)}{t^2} -2 \frac{f'(x_0)}{t}|\leq 4
\,\varepsilon .
$$

Now we have

$$
|f'(x_0)-\frac{f(x)-f(x_0)}{x-x_0}+\frac{1}{2}f''(x_0)(x-x_0)|=\frac{t}{2}\,
|f''(x_0)-2\frac{f(x_0+t)-f(x_0)}{t^2}+
$$
$$
2\frac{f'(x_0)}{t}|\leq \frac{t}{2} (
|f''(x_0)-2\frac{f(x_0+t)-f(x_0)}{t^2}+\frac{f(x_0+t)-
f(x_0-t)}{t^2}| +
$$
$$
|2 \frac{f'(x_0)}{t}- \frac{f(x_0+t)-f(x_0-t)}{t^2}|)\leq
\frac{t}{2}(\varepsilon+4\,\varepsilon)= \frac{5}{2}\,\varepsilon
|x-x_0|.
$$\end{proof}

Let $(X, |\cdot - \cdot|_X)$ and $(Y, |\cdot - \cdot|_Y)$ be
metric spaces. A mapping $f:X\to Y$ {\it satisfies Lipschitz
condition with a constant $C>0$} if $|f(x) - f(y)|_Y\leq C
|x-y|_X$ for any $x,y\in X$.  A mapping
$f:X\to Y$ is called {\it pointwise changeable} if for every
$\varepsilon>0$ the union $G_{\varepsilon}$ of  all open nonempty sets $G\subseteq X$ such
that $f|_G$ satisfies the Lipschitz condition with the constant
$\varepsilon$ is an everywhere dense set.

\begin{theorem}\label{t2}
Let $f:\mathbb R^2\to\mathbb R$ be a separately twice differentiable function such that
$f''_{xx}(p)=f''_{yy}(p)$ for every $p\in\mathbb R^2$. Then
the function $g:\mathbb R^2\to\mathbb R$, $g=f'_x-f'_y$, is pointwise changeable on every nonempty subset $E$ of arbitrary line $y=x+c$, and the function $h:\mathbb R^2\to\mathbb R$, $h=f'_x+f'_y$, is pointwise changeable on every nonempty subset $E$ of arbitrary line $y=-x+c$.
\end{theorem}

\begin{proof} Note that it ia sufficient to consider the case of closed set $E$ and $c=0$.

Let $\varepsilon>0$  and $E\subseteq \{(x,x):x\in\mathbb R\}$ be a closed set. According to Proposition \ref{p1}, Proposition \ref{p2}, Proposition \ref{p3} and Proposition \ref{p5} the functions $\lambda^{(2)}_1(\varepsilon,f^x)(y)$,
$\lambda^{(2)}_2(\varepsilon,f^x)(y)$,
$\lambda^{(2)}_1(\varepsilon,f_y)(x)$ and
$\lambda^{(2)}_2(\varepsilon,f_y)(x)$ are strictly positive and jointly upper semicontinuous. Therefore there exist $\delta>0$ and an open in $E$ set $G$ such that $|u-v|<\delta$
for every $(u,u), (v,v)\in G$ and the functions
$\lambda^{(2)}_1(\varepsilon,f^x)(y)$,
$\lambda^{(2)}_2(\varepsilon,f^x)(y)$,
$\lambda^{(2)}_1(\varepsilon,f_y)(x)$,
$\lambda^{(2)}_2(\varepsilon,f_y)(x)$ are $\ge\delta$ on $G$.

Let $(u,u), (v,v)\in G$, $u<v$ and $s=v-u$. According to Proposition \ref{p6}, there exist $\alpha_1, \alpha_2, \alpha_3, \alpha_4\in
(-\frac{5}{2}\varepsilon, \frac{5}{2}\varepsilon)$ such that
$$
f'_x(u,u)=\frac{f(v,u)-f(u,u)}{s}-\frac{1}{2}f''_{xx}(u,u)s+\alpha_1
s,
$$
$$
f'_y(u,u)=\frac{f(u,v)-f(u,u)}{s}-\frac{1}{2}f''_{yy}(u,u)s+\alpha_2
s,
$$
$$
f'_x(v,v)=\frac{f(u,v)-f(v,v)}{-s}-\frac{1}{2}f''_{xx}(v,v)(-s)+\alpha_3
s,
$$
$$
f'_y(v,v)=\frac{f(v,u)-f(v,v)}{-s}-\frac{1}{2}f''_{yy}(v,v)(-s)+\alpha_4
s.
$$
Then we have
$$
|g(u,u)-g(v,v)|=|f'_x(u,u)-f'_y(u,u)-f'_x(v,v)+f'_y(v,v)|=
$$
$$
|\alpha_1-\alpha_2-\alpha_3+\alpha_4|s< 10\,\varepsilon\,s.
$$

By analogy we can prove the pointwise changeability of the function $h=f_x'+f_y'$. \end{proof}

We need the following result from \cite{MM}.

\begin{proposition}\label{p7}\cite[Theorem 3.2]{MM}
Let $X\subseteq\mathbb R$ be a nonempty
interval, \mbox{$(Y,|\cdot - \cdot|_Y)$} be a metric space,
$f:X\to Y$ be a continuous pointwise changeable on every closed
set mapping. Then $f$ is a constant.
\end{proposition}

The following theorem will be proved using the idea of the proof of Theorem 4.3 from \cite{MM}.

\begin{theorem}\label{t3}
Let $f:\mathbb R^2\to\mathbb R$ be a separately twice differentiable function such that $f''_{xx}(p)=f''_{yy}(p)$
for every $p\in\mathbb R^2$, and assume that the function $g:\mathbb
R^2\to\mathbb R$, $g=f'_x-f'_y$, is jointly continuous. Then there exist continuously differentiable functions
$\varphi:\mathbb R \to \mathbb R$ and $\psi:\mathbb R \to \mathbb R$
such that $f(x,y)=\varphi(x-y)+\psi(x+y)$ for every
$x,y\in\mathbb R$.
\end{theorem}

\begin{proof} Theorem \ref{t2} and Proposition \ref{p7} imply that the function $g$ is a constant on every line $y=x+c$, i.e.
$g(x,y)=\alpha(x-y)$. The joint continuity of $g$ implies the continuity of the function $\alpha:\mathbb R\to\mathbb R$.

Choose a differentiable function $\varphi:\mathbb R\to\mathbb R$
such that $\varphi'(x)=\frac{1}{2}\alpha(x)$ for every
$x\in\mathbb R$ and put $h(x,y)=f(x,y)-\varphi(x-y)$. Then for every $x,y\in\mathbb R$ we have
$$\textstyle
h'_x(x,y)-h'_y(x,y)=f'_x(x,y)-\frac{1}{2}\alpha(x-y)-f'_y(x,y)-\frac{1}{2}\alpha(x-y)=
g(x,y)-\alpha(x-y)=0.
$$

Thus according to \cite[Corollary]{MM} there exists a continuous
$\psi:\mathbb R\to\mathbb R$ such that $h(x,y)=\psi(x+y)$, i.e. $f(x,y)=\varphi(x-y)+\psi(x+y)$ for every $x,y\in\mathbb R$.
Since $\psi(x)=f(x,0)-\varphi(x)$, $\psi$ is a continuously differentiable function.\end{proof}

Note that a similar result can be analogously proved under the assumption of continuity of the function $f'_x+f'_y$. It can be derived from Theorem \ref{t3} using the substitution of variables $x=u$, $y=-v$.

Note that in these results as in \cite[Theorem 4.1, Theorem 4.3, Corollary 4.4]{MM} the fact that $f(x,y)$
is defined on the hole plane $\mathbb R^2$ is inessential. The following result can be proved analogously with several formal changes.

\begin{theorem}\label{t4}
Let $P=\{(x,y)\in\mathbb R^2: a<x<b,\,
c<y<d\}$, $f:P\to\mathbb R$ be a function such that
$f''_{xx}(p)=f''_{yy}(p)$ for every $p\in P$, and either the function $g=f'_x-f'_y$, or the function $h=f'_x+f'_y$ is continuous on $P$.
Then there exist continuously differentiable functions $\varphi:(a,b) \to
\mathbb R$ and $\psi:(c,d) \to \mathbb R$ such that
$f(x,y)=\varphi(x-y)+\psi(x+y)$ for every $(x,y)\in P$.
\end{theorem}

\section{Properties of partial derivatives of separately twice differentiable functions}

In this section we establish some properties of partial derivatives of separately twice differentiable functions. The same property of separately differentiable functions and separately pointwise Lipschitz functions was obtained in \cite{Tol1, MM}.

\begin{theorem}\label{t5}
Let $X=Y=\mathbb R$ and $f:X\times
Y\to\mathbb R$ be a function which is twice differentiable with respect to the first variable $x$ and continuous with respect to the second variable $y$. Then for every nonempty set $E\subseteq X\times Y$ the restriction $f'_x|_E$ of the partial derivative $f_x$ on $E$ has a nowhere dense discontinuity points set.
\end{theorem}

\begin{proof} It is sufficient to consider the case of closed set $E$.

Let $G$ be a  nonempty open subset of $E$. We should find a nonempty open subset $W\subseteq G$ such that the restriction $f'_x|_W$ is continuous.

Note that according to Proposition \ref{p3} and Proposition \ref{p5} the function
$g(x,y)=\lambda^{(2)}_1(1,f_y)(x)$ is a strictly positive jointly upper semicontinuous function. In particular, the restriction
$g|_E$ is a strictly positive jointly upper semicontinuous function on a Baire space $E$. Therefore there exist $\delta>0$ and a  nonempty open subset $W\subseteq G$ such that $g(x,y)\geq
\delta$ for every $(x,y)\in W$.

Let us show that $f'_x|_E$ is continuous at every point
$(x_0,y_0)\in W$. Fix an $\varepsilon>0$. According to Baire theorem \cite{Baire} the separately continuous function $f$ has a dense $G_\delta$-set of joint continuity points on the set $\{(x,y_0):x\in X\}$. Therefore there exists $0<s<{\rm
min}\{\varepsilon, \delta\}$ such that $f$ is jointly continuous at the points $(x_0+s, y_0)$ and $(x_0-s, y_0)$.
Choose $0<\gamma<\frac{s}{2}$ such that
$$
\left|\frac{f(x_0+s,y_0)-f(x_0-s,y_0)}{2s}-\frac{f(x_2,y)-f(x_1,y)}{x_2-x_1}\right|<\varepsilon\eqno (4)
$$
\noindent for every $x_1,x_2\in X$
and $y\in Y$ with $|x_1-(x_0-s)|<\gamma$, $|x_2-(x_0+s)|<\gamma$ and
$|y-y_0|<\gamma$.

Let $(x',y')\in W$ with $|x'-x_0|<\gamma$ and $|y'-y_0|<\gamma$.
Since $s<\delta$ and $(x_0,y_0), (x',y')\in W$, according to the choice of $\delta$ and $W$ we have $s<
\lambda^{(2)}_1(1,f^{y_0})(x_0)$ ³
$s<\lambda^{(2)}_1(1,f^{y'})(x')$. Therefore the Proposition \ref{p4} implies
$$
\left|\frac{f(x_0+s,y_0)-f(x_0-s,y_0)}{2s} -
f'_x(x_0,y_0)\right|\leq 2s \eqno (5)
$$
\noindent and
$$
\left|\frac{f(x'+s,y')-f(x'-s,y')}{2s} - f'_x(x',y')\right|\leq
2s. \eqno (6)
$$
But $|(x'-s)-(x_0-s)|=|(x'+s)-(x_0+s)|=|x'-x_0|<\gamma$ and
$|y'-y_0|<\gamma$. Therefore according to (4)  we have
$$
\left|\frac{f(x_0+s,y_0)-f(x_0-s,y_0)}{2s} -
\frac{f(x'+s,y')-f(x'-s,y')}{2s}\right|\leq \varepsilon. \eqno (7)
$$
Taking to account (5), (6) and (7) we obtain
$$
|f'_x(x_0,y_0)-f'_x(x',y')|< \varepsilon + 4s < 5\varepsilon
$$
for every $(x',y')\in W$ with $|x'-x_0|<\gamma$ and
$|y'-y_0|<\gamma$. Thus $f'_x|_E$ is continuous at
$(x_0,y_0)$.\end{proof}

\section{Auxiliary propositions}

\begin{lemma}\label{l1}
Let $f:\mathbb R^2\to\mathbb R$ be a separately differentiable function, $\varphi:(a,b) \to \mathbb R$,
$\psi:(c,d) \to \mathbb R$  be continuously differentiable functions such that $f(x,y)=\varphi(x-y)+\psi(x+y)$ for $a<x-y<b$ and
$c<x+y<d$. Then there exists a continuously differentiable extension $\tilde{\psi}:[c,d] \to \mathbb R$ of $\psi$ such that
$f(x,y)=\varphi(x-y)+\tilde{\psi}(x+y)$,
$f'_x(x,y)=\varphi'(x-y)+\tilde{\psi}'(x+y)$ and
$f'_y(x,y)=-\varphi'(x-y)+\tilde{\psi}'(x+y)$ for $a<x-y<b$ and
$c\leq x+y\leq d$.
\end{lemma}

\begin{proof} Let $x_0, y_0 \in\mathbb R$ such that
$a<x_0-y_0<b$ and $x_0+y_0=d$. Then $f(x_0,y_0)=\lim\limits_{x\to
x_0-0}f(x,y_0)=\lim\limits_{x\to
x_0-0}(\varphi(x-y_0)+\psi(x+y_0))=\varphi(x_0-y_0)+\lim\limits_{t\to
d-0}\psi(t)$. Thus there exists $\tilde{\psi}(d)= \lim\limits_{t\to
d-0}\psi(t)$. Analogously there exists $\tilde{\psi}(c)= \lim\limits_{t\to
c+0}\psi(t)$. The extension
$\tilde{\psi}:[c,d] \to \mathbb R$ of $\psi$ is continuous and
$f(x,y)=\varphi(x-y)+\tilde{\psi}(x+y)$ for $a<x-y<b$ and $c\leq
x+y\leq d$.

Show that $\tilde{\psi}$ is continuously differentiable at $c$
and $d$. Let $x_0, y_0 \in\mathbb R$ such that $a<x_0-y_0<b$ and
$x_0+y_0=d$. Then $f'_x(x_0,y_0)=\varphi'(x_0-y_0)+
\lim\limits_{t\to d-0}\frac{\psi(t)-\tilde{\psi}(d)}{t-d}$. Thus
$\tilde{\psi}$ is differentiable at $d$ and
$f'_x(x_0,y_0)=\varphi'(x_0-y_0)+\tilde{\psi}'(d)$. The continuity of $f'_x$ with respect to $x$
implies the continuity of $\tilde{\psi}'$ at $d$.

The continuity of $\tilde{\psi}'$ at $c$
and the equalities $f'_x(x,y)=\varphi'(x-y)+\tilde{\psi}'(x+y)$ and
$f'_y(x,y)=-\varphi'(x-y)+\tilde{\psi}'(x+y)$ at the corresponding points $(x,y)$ can be obtained analogously.\end{proof}

\begin{lemma}\label{l2}
Let $r:(a,b)\to \mathbb R$ be a function such that
$r(w)\geq{\rm min}\{r(u),r(v)\}$ for $a<u<w<v<b$. Then there exists $c\in
[a,b]$ such that $r$ increases on the segment $(a,c)$ and decreases
on the segment $(c,b)$. Besides the set
$A=\{(u,v):u\in(a,b), \, v<r(u)\}$ is open if and only if
$r(u)\leq\lim\limits_{t\to u\pm 0}r(t)$ for every $u\in
(a,b)$.
\end{lemma}

\begin{proof} Let $a<u_1<v_1<b$ such that $r(u_1)< r(v_1)$.
Show that $r$ increases on the segment $(a,u_1]$.
Suppose that $a<u_2<v_2\leq u_1$ such that $r(u_2)> r(v_2)$.
Put $u=u_2$, $v=v_1$ and chose $w\in \{u_1, v_2\}$
such that $r(w)={\rm min}\{r(u_1), r(v_2)\}$. Now we have $r(w)<
{\rm min}\{r(u),r(v)\}$, which is impossible.

Now denote by $c$ the supremum of set of all $u\in
(a,b)$ such that $r$ increases on the segment $(a,u]$. If this set is empty then $c=a$.
It is clearly that $r$ increases on $(a,c)$. According to the chose of $c$, the function $r$ decreases on
$(c,b)$.

Prove the second part of Lemma. Fix $u\in (a,b)$ and suppose that $r(u)>\lim\limits_{t\to
u- 0}r(t)$. Put $v=\frac{1}{2}(r(u)+\lim\limits_{t\to u-
0}r(t))$. Then $(u,v)\in A$, but $A$ is not a neighborhood of $(u,v)$. Thus $A$ is not open.

Now let $v<r(u)$, i.e. $(u,v)\in A$. It follows from
$\lim\limits_{t\to u\pm 0}r(t) > v$ that $A$
is a neighborhood of $(u,v)$.\end{proof}

Note that for increasing function $r$ the above-mentioned condition means the left-continuity of $r$ and for
decreasing function $r$ it means the right-continuity of $r$.

\begin{lemma}\label{l3}
Let $r:(a,b)\to \mathbb R$ be a strictly increasing left-continuous function. Then there exist $u_0\in (a,b)$ and a strictly monotone sequence $(u_n)_{n=1}^{\infty}$ of points of the interval $(a,b)$ that converges to $u_0$ and satisfies the inequality $u_n-u_0\leq
r(u_n)-r(u_0)$ for every $n\in\mathbb N$.
\end{lemma}

\begin{proof} Suppose that it is impossible to choose such $u_0\in (a,b)$ and $(u_n)_{n=1}^{\infty}$. This means that for every $u\in(a,b)$ there exists $\delta_u>0$ such that $r(t)< t-u+r(u)$ for every
$t\in (u-\delta_u,u+\delta_u)\subseteq (a,b)$, $t\ne u$. Then $r(u)\leq\lim\limits_{t\to u+0}r(t)\leq
\lim\limits_{t\to u+0}(t-u+r(u))=r(u)$. Hence $r$
is a right-continuous function at every point $u\in (a,b)$. Thus $r$
is a continuous function.

Fix $u_1\in (a,b)$. Put
$u_2=u_1+\frac{1}{2}\delta_{u_1}$. According to the choice of $\delta_{u_1}$
we have $r(u_2)<u_2-u_1+r(u_1)$. Consider the set
$A=\{t\in[u_1,b): r(t)=t-u_2+r(u_2)\}$. Note that $A$
is a closed set, $u_1\not\in A$ and $u_2\in A$. Then $u_3=\inf A \in A$ and
$u_1<u_3\leq u_2$. The continuity of $r$ implies that
$r(t)>t-u_2+r(u_2)=t-u_3+r(u_3)$ for all $t\in [u_1,u_3)$. But it contradicts to the choice of $\delta_{u_3}$.\end{proof}

\begin{lemma}\label{l4}
Let $r:(a,b)\to (v_0,+\infty)$ be a strictly monotone function such that for the numbers
$v_1=\inf\limits_{u\in(a,b)}r(u)$, $v_2=\sup\limits_{u\in(a,b)}r(u)$ and a number $v_0<v_1$ the
the set
$A=\{(u,v):a<u<b,\,v_0<v<r(u)\}$ is open. Assume that $f:\mathbb R^2\to\mathbb R$ is a function such that for every rectangle
$P=(\alpha,\beta)\times (c,d)\subseteq A$ there exist continuously differentiable functions $\varphi_P:(\alpha,\beta) \to \mathbb R$ and $\psi_P:(c,d) \to \mathbb R$ such that
$f(u,v)=\varphi_P(u)+\psi_P(v)$ for every $(u,v)\in P$. Then
there exist continuously differentiable functions $\varphi:(a,b) \to
\mathbb R$ and $\psi:(v_0,v_2) \to \mathbb R$ such that
$f(u,v)=\varphi(u)+\psi(v)$ for all $(u,v)\in A$.
\end{lemma}

\begin{proof} Put $P_0=(a,b)\times(v_0,v_1)$ and
$\varphi=\varphi_{P_0}$. We shall show that for every rectangle
$P=(\alpha,\beta)\times (c,d)\subseteq A$ the function $\varphi_P-\varphi$
is constant on $(\alpha,\beta)$.

Firstly consider the case of $P=(\alpha,\beta)\times (v_0,v_1)$.
We have $\varphi(u)+\psi_{P_0}(v)=\varphi_P(u)+\psi_{P}(v)$, i.e.
$\varphi(u)-\varphi_P(u)=\psi_{P}(v)-\psi_{P_0}(v)$ for $(u,v)\in
P$. Therefore the functions $\varphi_P-\varphi$ and $\psi_{P}-\psi_{P_0}$ are constant on $(\alpha,\beta)$ and $(v_0,v_1)$ respectively.

Now let $P=(\alpha,\beta)\times (c,d)$,
$P_1=(\alpha,\beta)\times (v_0,d)$ and $P_2=(\alpha,\beta)\times
(v_0,v_1)$. Then, using analogous arguments, we obtain that the functions
$\varphi_P-\varphi_{P_1}$, $\varphi_{P_1}-\varphi_{P_2}$ and
$\varphi_{P_2}-\varphi$ are constant on $(\alpha,\beta)$. Therefore the function $\varphi_P-\varphi$ is constant on $(\alpha,\beta)$.

Without loss of generality we can suppose that
$\varphi_P-\varphi=0$ for every rectangle
$P=(\alpha,\beta)\times (c,d)\subseteq A$. Now by the similar arguments
we obtain that $\psi_{P_1}(v)=\psi_{P_2}(v)$ for every rectangles $P_1=(\alpha_1,\beta_1)\times (c_1,d_1)$,
$P_2=(\alpha_2,\beta_2)\times (c_2,d_2)\subseteq A$ and $v\in
(c_1,d_1)\cap(c_2,d_2)$. For every $v\in (v_0,v_2)$
choose $u\in (a,b)$ such that $(u,v)\in A$. Choose any open neighborhood
$P=(\alpha,\beta)\times (c,d)\subseteq A$ of $(u,v)$ and put $\psi(v)=\psi_P(v)$.\end{proof}

The following property of solutions of the equation $(3)$ plays an important role in the proof of the main result.

\begin{theorem}\label{t6}
Let $f:\mathbb R^2\to\mathbb R$ be a separately twice differentiable function such that $f''_{xx}(p)=f''_{yy}(p)$ for every $p\in\mathbb R^2$, $g=f'_x-f'_y$, $W=\{(x,y)\in\mathbb R^2: a\leq x-y\leq b, \, c\leq
x+y\leq d\}$ be a rectangle and $I$ be a non-empty open subset of $g(W)\subset\mathbb R$ such that $g$ is continuous at every point of the set
$E=g^{-1}(I)\cap W$. Then there exists a nonempty open segment
$(a_0,b_0)\subseteq (a,b)$ such that $\{(x,y)\in\mathbb R^2: a_0<
x-y< b_0, \, c\leq x+y\leq d\}\subseteq E$.
\end{theorem}

\begin{proof} Since $I$ is open in $g(W)$ and $g$ is continuous at every point of the set $E$, the set $E$ is open in $W$.

Put $u=x-y$ and $v=x+y$. Fix $(x_0,y_0)\in
E$. Since $E$ is open in $W$, there exists a nonempty closed segment $[a_1,b_1]\subseteq (a,b)$ such that $B=\{(x,y):a_1\leq
u\leq b_1, \, v=v_0=x_0+y_0\}\subseteq E$.

We shall show that there exists a nonempty open segment $(a_2,b_2)\subseteq
(a_1,b_1)$ such that $\{(x,y):a_2< u< b_2, \, v_0\leq v\leq
d\}\subseteq E$. Since $E$ is open in $W$, it is sufficient to prove that there exists $w\in(a_1,b_1)$ such that the compact set $\{(x,y):u=w, \, v_0\leq v\leq d\}$ is contained in $E$.

Assume the contrary, i.e. $\{(x,y): u=w, \,v_0\leq v\leq d\}\not\subseteq E$ for every $w\in (a_1,b_1)$. For arbitrary
$w\in (a_1,b_1)$ and $t\in(v_0,d]$ put $J_{w,t}=\{(x,y):u=w,\,
v_0\leq v<t\}$. For every $u\in (a_1,b_1)$ denote by
$r(u)$ the supremum of the set of all $v\in(v_0,d]$ such that $g$ is continuous on $J_{u,v}$. Since $E$ is a neighborhood of compact set $B$ in $W$ and $g$ is continuous at every point $p\in E$, the function $r:(a_1,b_1)\to\mathbb R$
is correctly defined and there exists $v_1\in(v_0,d]$ such that
$r(u)\geq v_1$ for every $u\in (a_1,b_1)$.

Let us show that $r$ satisfies the conditions of Lemma \ref{l2} on the segment $(a_1,b_1)$. Suppose the contrary, i.e. there exist $a_1<u_1<u_2<u_3<b_1$ such that
$r(u_2)<\min\{r(u_1),\,r(u_3)\}$. Put
$d_1=\inf\limits_{u\in(u_1,u_3)}r(u)$. Note that $v_0<v_1\leq
d_1\leq r(u_2)$. Since $g$ is continuous on each set
$J_{u,d_1}$ for $u\in(u_1,u_3)$, according to Theorem \ref{t2} and Proposition \ref{p7} $g$ is a constant on $J_{u,d_1}$. Since $B\subseteq E$ and $J_{u,d_1}\cap B\ne \O$,  $J_{u,d_1}\subseteq E$ for every $u\in (u_1,u_3)$. Therefore
$g$ is continuous at every point of the rectangle
$P=\bigcup\limits_{u\in(u_1,u_3)}J_{u,d_1}=\{(x,y):u_1<x-y<u_3,\,
v_0<x+y<d_1\}$.

According to Theorem \ref{t4}, there exist continuously differentiable functions
$\varphi:(u_1,u_3) \to \mathbb R$ and $\psi:(v_0,d_1) \to \mathbb R$
such that $f(x,y)=\varphi(x-y)+\psi(x+y)$ for every $(x,y)\in
P$. It follows from Lemma \ref{l1} that $g(x,y)=2\varphi'(x-y)$ for
$x+y=d_1$ and $x-y\in(u_1,u_3)$. Thus $g$ is  constant on every segment $J_w=\{(x,y):u=w, v_0\leq v\leq d_1\}$ for $w\in
(u_1,u_3)$. On the other hand, since $r(u_1)>r(u_2)\geq d_1$ and
$r(u_3)>r(u_2)\geq d_1$, $g$ is continuous on $J_{u_1}$ and $J_{u_3}$. Thus $g$ is constant on $J_{u_1}$ and $J_{u_3}$. Analogously as for the rectangle $P$ we obtain that $P_1=\{(x,y):u_1\leq u\leq u_3, \,
v_0\leq v\leq d_1\}\subseteq E$. Since $E$ is a neighborhood of the compact rectangle $P_1$ in $W$ and $d_1< r(u_3)<d$, there exists
$v_2\in(d_1,d)$ such that $W_1=\{(x,y)\in\mathbb R^2: u_1\leq u\leq
u_3, \, v_0\leq v\leq v_2\}\subseteq E$. In particular, $g$ is continuous at every point $p\in W_1$. Then $r(u)\geq v_2$ for every
$u\in[u_1,u_3]$. But this contradicts to the choice of $d_1$.

Thus  $r$ satisfies the conditions of Lemma \ref{l2} on the segment $(a_1,b_1)$. Therefore there exists an open segment
$(a_3,b_3)\subseteq (a_1,b_1)$ such that $r$ is monotone on $(a_3,b_3)$.

We claim that $r$ is constant on no open segment. Assume that $r(u)=d_0$ for all $u\in(a',b')\subseteq (a_1,b_1)$. Then according to Theorem \ref{t2} and Proposition \ref{p7} for every $u\in(a',b')$ the function $g$ is constant on the segment $J_{u,d_0}$ and equals to the value of $g$ at the point $p_u=(\frac{1}{2}(v_0-u),\frac{1}{2}(v_0+u))$. Since $p_u\in B\subseteq E$, $J_{u,d_0}\subseteq E$ for every $u\in(a',b')$. Thus $g$ is continuous on the rectangle $\bigcup\limits_{u\in(a',b')}J_{u,d_0}$. It follows from Theorem \ref{t4} and Lemma \ref{l1} that $g(\frac{1}{2}(d_0-u),\frac{1}{2}(d_0+u))=g(\frac{1}{2}(v_0-u),\frac{1}{2}(v_0+u))$ for every $u\in(a',b')$. Therefore $g$ is continuous at $(\frac{1}{2}(d_0-u),\frac{1}{2}(d_0+u))$, what contradicts to $r(u)=d_0$.

Hence $r$ is not constant on every open segment. Therefore $r$ is a strictly monotone function on $(a_3,b_3)$. Without loss of the generality we can suppose that $r$ strictly increases on $(a_3,b_3)$.
Since $A=\{(x,y):a_3<u<b_3,\,v_0<v<r(u)\}\subseteq
E$, $g$ is continuous at every point from open set $A$. Therefore according to Theorem \ref{t4} the functions $r:(a_3,b_3)\to\mathbb R$ and $\tilde{f}:\mathbb R^2\to\mathbb R$, $\tilde{f}(u,v)=f(x,y)$,
satisfy the conditions of Lemma \ref{l4}. Therefore there exist continuously differentiable functions $\varphi:(a_3,b_3) \to \mathbb R$ and $\psi:(v_0,v_3) \to \mathbb R$, where
$v_3=\sup\limits_{u\in(a_3,b_3)}r(u)$, such that
$f(x,y)=\varphi(x-y)+\psi(x+y)$ for every $(x,y)\in A$.

Let us show that $\tilde{f}(u,r(u))=\varphi(u)+\psi(r(u))$ for every
$u\in(a_3,b_3)$. Since $r$ strictly increases, $r(u)<v_3$, i.e. the function $\psi$ is defined at $r(u)$. Fix any strictly decreasing sequence $(w_n)_{n=1}^{\infty}$ of points
$w_n\in(a_3,b_3)$ which converges to $u$. Note that $w_n>u$. Therefore
$z_n=-w_n+u+r(u)<r(u)<r(w_n)$. Thus
$p_n=(\frac{1}{2}(w_n+z_n),\frac{1}{2}(z_n-w_n))=
(\frac{1}{2}(u+r(u)),\frac{1}{2}(u+r(u))-w_n)\in A$ for every
$n\in\mathbb N$. It follows from the continuity of $f$  with respect to $y$ that

$$\tilde{f}(u,r(u))=
\lim\limits_{n\to\infty}f(\frac{1}{2}(u+r(u)),\frac{1}{2}(u+r(u))-w_n)
=$$

$$\lim\limits_{n\to\infty}(\varphi(w_n)+ \psi(u+r(u)-w_n))=\varphi(u) + \psi(r(u)).$$

Besides, for the point $p=(\frac{1}{2}(u+r(u)),\frac{1}{2}(r(u)-u))$
we have
$$
f_y'(p)=\lim\limits_{n\to\infty}\frac{f(p_n)-f(p)}{u-w_n}=-\varphi'(u)+\psi'(r(u)).
$$

According to Lemma \ref{l3} choose $u_0\in (a_3,b_3)$
and a convergent to $u_0$ strictly monotone sequence
$(u_n)_{n=4}^{\infty}$ of points $u_n\in(a_3,b_3)$ such that
$v_n=u_n-u_0+r(u_0)\leq r(u_n)$ for $n\geq 4$. Since
$\tilde{f}(u,v)=\varphi(u)+\psi(v)$ for $a_3<u<b_3$ and $v_0<v\leq
r(u)$, $f(\frac{1}{2}(v_n+u_n),
\frac{1}{2}(v_n-u_n))=f(u_n+\frac{1}{2}(r(u_0)-u_0),
\frac{1}{2}(r(u_0)-u_0))=\tilde{f}(u_n,v_n)=\varphi(u_n)+\psi(v_n)$.

Put
$p_0=(\frac{1}{2}(r(u_0)+u_0),\frac{1}{2}(r(u_0)-u_0))$. We have

$$
f_x'(p_0)=\lim\limits_{n\to\infty}\frac{\tilde{f}(u_n,v_n)-\tilde{f}(u_0,r(u_0))}{u_n-u_0}=
$$
$$
=\lim\limits_{n\to\infty}\frac{\varphi(u_n)-\varphi(u_0)+\psi(u_n-u_0+r(u_0))-\psi(r(u_0))}{u_n-u_0}=\varphi'(u_0)+\psi'(r(u_0)).
$$

Thus $g(p_0)=f'_x(p_0)-f'_y(p_0)=2\varphi'(u_0)$. Hence $g$ is constant on the set $\{(x,y):u=u_0,\,v_0\leq v\leq
r(u_0)\}\subseteq E$. Since $r(u_0)<d$ and $E$ is open in $W$, there exists $\tilde{v}\in(r(u_0),d)$ such that $\{(x,y):u=u_0,\,v_0\leq
v\leq \tilde{v}\}\subseteq E$. But this contradicts the choice of
$r(u_0)$.

Thus there exists a nonempty open segment $(a_2,b_2)\subseteq(a_1,b_1)$
such that $\{(x,y):a_2<u<b_2,\,v_0\leq v\leq d\}\subseteq E$.
Using analogous arguments we obtain that there exists a nonempty open segment
$(a_0,b_0)\subseteq (a_2,b_2)$ such that
$\{(x,y):a_0<u<b_0,\,c\leq v\leq v_0\}\subseteq E$.\end{proof}

For a real-valued function $f:X\to
\mathbb R$ defined on a topological space $(X,\tau)$ and a point $x_0\in X$ let
$$\omega_f(x_0)=\inf\{\omega_f(U):x_0\in U\in\tau\}$$be the {\em oscillation of $f$ at $x_0$}, where $$\omega_f(U)=\mathrm{diam} f(U)=\sup\limits_{x',x''\in U}|f(x')-f(x'')|$$is the oscillation of $f$ on a subset $U\subset X$.

\begin{corollary}\label{c1}
Let $f:\mathbb R^2\to\mathbb R$ be a separately twice differentiable function such that
$f''_{xx}(p)=f''_{yy}(p)$ for every $p\in\mathbb R^2$, $D$ be the set of discontinuity points of the function $g=f'_x-f'_y$ and $W=\{(x,y)\in\mathbb R^2: a< x-y< b, \,
c< x+y<d\}$ be a rectangle such that the restriction $g|_{D\cap W}$ is continuous. Then the projection $\pi(D\cap W)$, where $\pi:\mathbb
R^2\to\mathbb R$, $\pi(x,y)=x-y$, is a meager subset of $\mathbb R$.
\end{corollary}

\begin{proof} For every $n\in\mathbb N$ put $D_n=\{p\in
D\cap W: \omega_g(p)\geq\frac{1}{n}\}$. Note that it is sufficient to prove that all sets $\pi(D_n)$ are meager.

Fix $n\in\mathbb N$. For every $p\in D_n$, use the continuity of the restriction $g|_{D\cap W}$  to choose a neighborhood $W_p=\{(x,y)\in \mathbb R^2: a_p\leq x-y\leq b_p,\,c_p\leq
x+y\leq d_p\}$ of $p$ such that $W_p\subseteq W$ and $\omega_g(D\cap
W_p)\leq \frac{1}{2n}$. Put $\widetilde{W}_p=\{(x,y)\in \mathbb
R^2: a_p< x-y< b_p,\,c_p< x+y< d_p\}$ and show that the set
$A_p=\pi(D_n\cap \widetilde{W}_p)$ is nowhere dense.

Suppose that the set $A_p$ is dense on a closed segment
$[\alpha,\beta]\subseteq (a_p,b_p)$. Put $W'=\{(x,y)\in
\mathbb R^2: \alpha\leq x-y\leq \beta,\,c_p\leq x+y\leq d_p\}$,
$I=\mathbb R\setminus \overline{g(D\cap W_p)}$ and pick any point $q\in D_n\cap \{(x,y)\in\mathbb R^2: \alpha< x-y< \beta, \,
c_p< x+y< d_p\}$. Since ${\rm diam}(\overline{g(D\cap
W_p)})\leq \frac{1}{2n}$ and ${\rm diam}(g(W'))=\omega_g(W')\geq
\omega_g(q)\geq \frac{1}{n}$, the set
$E=g^{-1}(I)\cap W'$ is nonempty. Besides, $E\cap D\subseteq
g^{-1}(I)\cap W_p\cap D=\O$. Therefore $g$ is continuous at every point from
$E$. Thus the function $g$ satisfies the conditions of Theorem \ref{t6} on the rectangle $W'$. Therefore there exists a nonempty open segment $(\alpha_0,\beta_0)\subseteq (\alpha,\beta)$ such that $\{(x,y)\in
\mathbb R^2: \alpha_0< x-y< \beta_0,\,c_p\leq x+y\leq
d_p\}\subseteq E$. It implies $A_p\cap(\alpha_0,\beta_0)=\O$, what contradicts to the density of $A_p$ on $[\alpha,\beta]$.

Hence all the sets $A_p$ are nowhere dense. Using the
$\sigma$-compactness of $D_n$ (see \cite{E} for definition) choose a sequence
$(p_m)_{m=1}^{\infty}$ of points $p_m\in D_n$ such that $D_n\subseteq
\bigcup\limits_{m=1}^{\infty}\widetilde{W}_{p_m}$. Then $\pi(D_n)\subseteq \bigcup\limits_{m=1}^{\infty} A_{p_m}$ and
$\pi(D_n)$ is a meager set.\end{proof}

\begin{remark}\label{r2}
The analogous result for the functions $h=f'_x+f'_y$ and $\pi(x,y)=x+y$ is valid too.
\end{remark}

The following statement will be used in the final stage of our reasoning.

\begin{theorem}\label{alpha} A continuous function $\alpha:\mathbb R\to\mathbb R$ is differentiable if for every $y\in\mathbb R$ the function $\beta_y(x)=\alpha(x)-\alpha(x+y)$ is differentiable.
\end{theorem}

\begin{proof} Assume that the function $\alpha$ is not differentiable but for every $y$ the function $\beta_y$ is differentiable.

\begin{claim}\label{cl1} The function $\alpha$ is nowhere differentiable.
\end{claim}

\begin{proof} Assume that $\alpha$ is differentiable at some point $x_0$. Then for every $y\in\mathbb R$ the function $\alpha_y(x)=\alpha(x) - \beta_y(x) =\alpha(x+y)$ is differentiable at $x_0$. Hence $\alpha$ is differentiable at $x_0+y$ for every $y\in\mathbb R$.
\end{proof}

For every $n\in\mathbb N$ put $$A_n=\{x\in\mathbb R:\forall x'\in(x,x+\tfrac1n)\;\; |\alpha(x)-\alpha(x')|\le n|x-x'|\}.$$

\begin{claim}\label{cl2} For every $n\in\mathbb N$ the set $A_n$ is closed and nowhere dense in $\mathbb R$.
\end{claim}

\begin{proof} The continuity of $\alpha$ implies the closedness of every $A_n$. Now assume that $A_n$ contains some interval $(a,b)$. Then $\alpha$ is Lipschitz on $(a,b)$ and hence is differentiable at some point, which is forbidden by Claim~\ref{cl1}.
\end{proof}

Claim~\ref{cl2} implies:

\begin{claim} The union $A_\infty=\bigcup_{n=1}^\infty A_n$ is of the first Baire category in $\mathbb R$.
\end{claim}

For every $n\in\mathbb N$ put $$B_n=\{(x,y)\in\mathbb R^2:\forall x'\in(x,x+\tfrac1n)\;\; |\beta_y(x)-\beta_y(x')|\le n|x-x'|\}.$$
Since the function $f(x,y)=\beta_y(x)$ is jointly continuous, every set $B_n$ is closed. Since the function $\beta_y$ is differentiable for each $y$, we conclude that $\mathbb R^2=\bigcup\limits_{n=1}^\infty B_n$ and by the Baire Theorem, for some $n_0\in\mathbb N$ the set $B_{n_0}$ has non-empty interior and hence contains some rectangle $(a,b)\times(c,d)$.

Since the set $A_\infty=\bigcup_{n=1}^\infty A_n$ is of the first Baire category, there is a point $x_0\in (a,b)\setminus A_\infty$. Now consider the function $\alpha$ restricted to the interval $(x_0+c,x_0+d)$.

Since the function $\alpha$ is nowhere differentiable, it is not monotone on $(x_0+c,x_0+d)$. Consequently, for the continuous function $\alpha$ there exists $z_0\in\mathbb R$ such that the set $C=\{x\in(x_0+c,x_0+d): \alpha(x)=z_0\}$ contains at least two points. Since $\alpha$ is not constant on every interval, there exist $u,v\in C$ such that $u<v$ and $(u,v)\cap C=\O$. It follows from the continuity of $\alpha$ that $\alpha(t)>z_0$ for all $t\in(u,v)$ or $\alpha(t)<z_0$ for all $t\in(u,v)$. Put $\varepsilon=v-u$.

\begin{claim}\label{cl4} For every $\delta\in(0,\varepsilon)$ there is a point $t\in (u,v-\delta)$ such that $\alpha(t)=\alpha(t+\delta)$.
\end{claim}

\begin{proof} It is sufficient to use the Mean Value Theorem for the function $g:[u,v-\delta]\to\mathbb R$, $g(x)=\alpha(x)-\alpha(x+\delta)$.
\end{proof}

Choose $m\geq n_0$ so large that $\frac1m\le \min\{b-x_0,\varepsilon\}$. Since $x_0\notin A_m$, there is a point $x_1\in (x_0,x_0+\frac1m)$ such that $|\alpha(x_0)-\alpha(x_1)|>m|x_0-x_1|$. For $\delta=x_1-x_0$ by Claim~\ref{cl4}, there is a point $t\in (u,v-\delta)$ such that $\alpha(t)=\alpha(t+\delta)$. Let $y_0=t-x_0$ and observe that $y_0\in (u-x_0,v-x_0)\subseteq (c,d)$.   Then $(x_0,y_0)\in (a,b)\times(c,d)\subseteq B_{n_0}$. On the other hand,
$$
|\beta_{y_0}(x_0)-\beta_{y_0}(x_1)|=|\alpha(x_0)-\alpha(x_0+y_0)-\alpha(x_1)+\alpha(x_1+y_0)|=$$
$$=|\alpha(x_0)-\alpha(t)-\alpha(x_1)+\alpha(t+\delta)|=|\alpha(x_0)-\alpha(x_1)|>m|x_0-x_1|\geq n_0|x_0-x_1|$$which contradicts to $(x_0,y_0)\in B_{n_0}$.
\end{proof}

\section{Main result}

In this section we prove the main result of the paper.
We first prove the following auxiliary facts.

\begin{proposition}\label{p8}
Let $X$ be a second countable topological space, $Y$ be a Baire space and
$G\subseteq W\subseteq X\times Y$ be open sets such that
$W\subseteq \overline{G}$. Then there exists a
dense $G_{\delta}$-set $B\subseteq Y$ such that for every $y\in B$
the set $G_y=\{x\in X: (x,y)\in G\}$ is dense in $W_y=\{x\in
X: (x,y)\in W\}$.
\end{proposition}

{\bf Prof.} Put $G_0=((X\times Y)\setminus
\overline{W})\cup G$ and fix a base
$(U_n)_{n=1}^{\infty}$ of topology of $X$. Since $G$ is dense in $W$, the set $G_0$ is dense in $X\times Y$. For every
$n\in\mathbb N$ denote by $B_n$ the set of all $y\in Y$
such that there exist a neighborhood $V$ of $y$ in $Y$ and an open in $X$
nonempty set $U\subseteq U_n$ such that $U\times V\subseteq
G_0$. Clearly, all the sets $B_n$ are open. Since $G_0$ is dense in $X\times Y$, $B_n$ is dense in $Y$. Thus
$B=\bigcap\limits_{n=1}^{\infty} B_n$ is a dense in $Y$
$G_{\delta}$-set such that for every $y\in B$ the set
$\{x\in X: (x,y)\in G_0\}$ is dense in $X$. Therefore for every $y\in B$
the set $G_y$ is dense in $W_y$.\end{proof}

\begin{proposition}\label{p9}
Let $W=\{(x,y)\in\mathbb
R^2:a<x-y<b,\,c<x+y<d\}$, $f:W\to \mathbb R$, for every $p\in W$ there exist the partial derivatives $f'_x(p)$ and $f'_y(p)$ which are continuous with respect to $x$ and $y$ respectively and $G$ 
be a dense open subset in $W$ such that $f'_x(p)=f'_y(p)=0$ for every $p\in G$. Then $f$ is constant.
\end{proposition}

\begin{proof} Using Proposition \ref{p8} choose sets $A\subseteq
(\frac{a+c}{2},\frac{b+d}{2})$ and $B\subseteq
(\frac{c-a}{2},\frac{d-b}{2})$ such that $A$ and $B$ are dense in $(\frac{a+c}{2},\frac{b+d}{2})$ and $(\frac{c-a}{2},\frac{d-b}{2})$, respectively, for every $x\in A$
the set $G^x=\{y\in\mathbb R: (x,y)\in G\}$ is dense in $W^x=\{y\in\mathbb R: (x,y)\in W\}$ and for every $y\in B$
the set $G_y=\{x\in\mathbb R: (x,y)\in G\}$ is dense in $W_y=\{x\in\mathbb R: (x,y)\in W\}$.

Fix any $y\in B$. Since $f'_x$ is continuous on $W_y\times\{y\}$ and $f'_x(p)=0$ for every $p\in G_y\times\{y\}$, $f'_x=0$ on $W_y\times\{y\}$. Thus for every $y\in B$ the function $f$ is constant on $W_y\times\{y\}$.

Analogously for every $x\in A$ the function $f$ is a constant on
$\{x\}\times W^x$. Since $A$ and $B$ are dense in $(\frac{a+c}{2},\frac{b+d}{2})$
and $(\frac{c-a}{2},\frac{d-b}{2})$ respectively, the function $f$ is a constant on the set $(\bigcup\limits_{x\in A}(\{x\}\times W^x))\bigcup
(\bigcup\limits_{y\in B}(W_y\times\{y\}))$. The continuity of $f$ with respect to each variable implies that $f$ is constant on $W$.\end{proof}

\begin{theorem}\label{t7}
Let $f:\mathbb R^2\to\mathbb R$ be a separately twice differentiable function such that
$f''_{xx}(p)=f''_{yy}(p)$ for every $p\in\mathbb R^2$. Then there exist twice differentiable functions $\varphi:\mathbb R \to
\mathbb R$ and $\psi:\mathbb R \to \mathbb R$ such that
$f(x,y)=\varphi(x-y)+\psi(x+y)$ for every $x,y\in \mathbb R$.
\end{theorem}

\begin{proof} Firstly we prove that there exist continuously differentiable functions $\varphi:\mathbb R \to
\mathbb R$ and $\psi:\mathbb R \to \mathbb R$ such that
$f(x,y)=\varphi(x-y)+\psi(x+y)$ for every $x,y\in \mathbb R$.

The cases of continuous function $g=f'_x-f'_y$ and continuous function
$h=f'_x+f'_y$ are considered in Theorem \ref{t4}.

Let $U=V=\mathbb R$, $U\ni u=x-y$ and $V\ni v=x+y$ for $x\in
X=\mathbb R$ and $y\in Y=\mathbb R$, $\pi_u:U\times V\to U$,
$\pi_u(u,v)=u$, $\pi_v:U\times V\to V$, $\pi_v(u,v)=v$, and the function
$\tilde{f}, \tilde{g}, \tilde{h}:U\times V\to\mathbb R$ are defined by
equalities: $\tilde{f}(u,v)=f(x,y)$,
$\tilde{g}(u,v)=g(x,y)$ ³ $\tilde{h}(u,v)=h(x,y)$. Assume that $g$ and $h$ are discontinuous, i.e. the sets $D_1$ and
$D_2$ of points of discontinuity of $\tilde{g}$ and $\tilde{h}$ are nonempty.

We claim that there exist nonempty open segments $(a_1,b_1)\subseteq U$ and
$(c_1,d_1)\subseteq V$ and dense subsets $A\subseteq (a_1,b_1)$
and $B\subseteq (c_1,d_1)$ which satisfy the following conditions:

$(a)$\,\,\,$D\cap W_1\ne\O$, where $W_1=(a_1,b_1)\times (c_1,d_1)$ and
$D=D_1\cup D_2$;

$(b)$\,\,\,for every $u\in A$ the function $\tilde{g}$ is constant on $\{u\}\times (c_1,d_1)$;

$(c)$\,\,\,for every $v\in B$ the function $\tilde{h}$ is constant on $(a_1,b_1)\times \{v\}$.

It follows from Theorem \ref{t5} that the restriction $\tilde{g}|_{D_1}$ has nowhere dense discontinuity points set. Since $D_1\ne\O$, there exists an open in $U\times V$ set
$\tilde{W}_1=(\tilde{a}_1,\tilde{b}_1)\times(\tilde{c}_1,\tilde{d}_1)$
such that $D_1\cap \tilde{W}_1\ne\O$ and the function $\tilde{g}|_{D_1}$ is continuous at every point $p\in D_1\cap \tilde{W}_1$. According to Corollary \ref{c1}, the set $\pi_u(D_1\cap \tilde{W}_1)$ is a meager set.

Note that $D_2\cap \tilde{W}_1\ne\O$. Indeed, if $\tilde{h}$
is continuous on $\tilde{W}_1$, then according to Theorem \ref{t4} the function
$\tilde{g}$ is continuous on $\tilde{W}_1$.

Analogously using Remark \ref{r2} we find a rectangle $W_1=(a_1,b_1)\times(c_1,d_1)\subseteq \tilde{W}_1$ such that
$D_2\cap W_1\ne \O$ and the set $\pi_v(D_2\cap {W}_1)$ is meager. Clearly, $W_1$ satisfies $(a)$.
Put $A=(a_1,b_1)\setminus \pi_u(D_1\cap {W}_1)$ and
$B=(c_1,d_1)\setminus \pi_v(D_2\cap {W}_1)$. The conditions $(b)$ and $(c)$ follow from Theorem \ref{t2} and Proposition \ref{p7}.

Note that according to Theorem \ref{t4}, the condition $(a)$ is equivalent to
$D_1\cap W_1\ne\O$ and $D_2\cap W_1\ne\O$. Pick any $p\in D_1\cap W_1$. Since $W_1\subseteq\tilde{W}_1$, the function
$\tilde{g}|_{D_1}$ is continuous at $p$. Therefore there exists a neighborhood
$\tilde{W}_2=[\tilde{a}_2, \tilde{b}_2]
\times[\tilde{c}_2,\tilde{d}_2]\subseteq W_1$ of $p$ such that
$\omega_{\tilde{g}}(\tilde{W}_2\cap D_1)\leq
\frac{1}{2}\omega_{\tilde{g}}(p)$. Reasoning analogously as in the proof of Corollary \ref{c1}, according to Theorem \ref{t6}, find an open nonempty segment $(a'_2,b'_2)\subseteq (\tilde{a}_2, \tilde{b}_2)$ and continuously differentiable functions $\varphi_1:(a'_2,b'_2)\to\mathbb R$ and
$\psi:(\tilde{c}_2,\tilde{d}_2)\to\mathbb R$ such that
$\tilde{f}(u,v)=\varphi_1(u)+\psi(v)$, in particular,
$\tilde{h}(u,v)=2\psi'(v)$ for every $u\in (a'_2,b'_2)$ and
$v\in (\tilde{c}_2,\tilde{d}_2)$. According to $(c)$ we have that $\tilde{h}(u,v)=2\psi'(v)$ for every $u\in
(a_1,b_1)$ and $v\in B\cap(\tilde{c}_2,\tilde{d}_2)$.

Analogously we find an open neighborhood
$W_2=(a_2,b_2)\times(c_2,d_2)\subseteq \tilde{W}_2$ of some point
$q\in D_2$ and continuously differentiable function
$\varphi:(a_2,b_2)\to\mathbb R$ such that
$\tilde{g}(u,v)=2\varphi'(u)$ for every $u\in A\cap(a_2,b_2)$
and $v\in(c_1,d_1)$.

Consider the function $f_0:W_2\to\mathbb R$,
$f_0(u,v)=\varphi(u)+\psi(v)$. According to Theorem \ref{t5} choose an open dense in $W_2$ set $G\subseteq W_2$ such that the functions $\tilde{g}$ and $\tilde{h}$ are continuous at every point from $G$.
Pick a family $(P_s:s\in S)$ of rectangles
$P_s=(a_s,b_s)\times(c_s,d_s)\subseteq U\times V$ such that
$G=\bigcup\limits_{s\in S} P_s$. It follows from Theorem \ref{t4} that for every $s\in S$ there exist continuously differentiable functions $\varphi_s:(a_s,b_s)\to\mathbb R$ and $\psi_s:(c_s,d_s)\to\mathbb R$
such that $\tilde{f}(u,v)=\varphi_s(u)+\psi_s(v)$ for every $(u,v)\in P_s$. Taking into account that $(a_s,b_s)\subseteq
(a_2,b_2)\subseteq (a_1,b_1)$ and $(c_s,d_s)\subseteq
(\tilde{c}_2,\tilde{d}_2)\subseteq (c_1,d_1)$ we obtain that
$\varphi'_s(u)=\varphi'(u)$ for every $u\in A\cap(a_s,b_s)$ and
$\psi'_s(v)=\psi'(v)$ for every $v\in B\cap(c_s,d_s)$. Since the functions $\varphi'_s$, $\psi'_s$, $\varphi'$ and $\psi'$ are continuous and the sets $A$ and $B$ are dense in $(a_s,b_s)$ and $(c_s,d_s)$
respectively, the functions $\varphi_s-\varphi$ and $\psi_s-\psi$ are constant on $(a_s,b_s)$ and $(c_s,d_s)$ respectively. Thus the function
$\tilde{f}-f_0$ is constant on every rectangle $P_s$.

Now for the function $\gamma:Z\to\mathbb R$,
$\gamma(x,y)=f(x,y)-\varphi(x-y)-\psi(x+y)$, defined on the set
$Z=\{(x,y)\in X\times Y: (u,v)\in W_2\}$, we have
$\gamma'_x(p)=\gamma'_y(p)=0$ for every $p=(x,y)$ such that
$(u,v)\in G$. Therefore according to Proposition \ref{p9} the function $\gamma$ is constant. This implies that $\tilde{h}$ is a constant on
$W_2$. But it contradicts to $q\in D_2\cap W_2$.

Thus $f(x,y)=\varphi(x-y)+\psi(x+y)$ for some continuously differentiable functions $\varphi$ and $\psi$. It remains to prove that $\varphi$ and $\psi$ are twice differentiable.

Put $\alpha(x)=\varphi'(x)$ and $\beta(x)=\psi'(x)$. Note that $f'_x(x,y)=\alpha(x-y)+\beta(x+y)$. Since $f$ is twice differentiable with respect to $x$, for every $a\in\mathbb R$ the function $\gamma_a(x)=\alpha(x)+\beta(x+a)$ is differentiable. Therefore for every $a\in\mathbb R$ the function $\theta_a(x)=\alpha(x)-\alpha(x+a)=\gamma_a(x)-\gamma_0(x+a)$ is differentiable. Hence according to Theorem~\ref{alpha}, the function $\alpha$ is differentiable. The differentiability of the function $\beta$ can be proved by analogy.\end{proof}

\small{

}


\begin{thebibliography}{7}

\bibitem{Baire} R.~Baire, {\em Sur les fonctions de variables r\'eelles}, Annali di mat. pura ed appl., ser.{\bf 3} (1899) 1--123.
    
\bibitem{BPPT} A.M.~Bruckner, G.~Petruska, O.~Preiss, B.S.~Thomson, {\em The
equation $u_xu_y=0$ factors}, Acta Math. Hung. {\bf 57}:3-4 (1991) 275--278.

\bibitem{Cher} P.R.~Chernoff, H.F.~Royden, {\em The Equation
$\frac{\partial f}{\partial x} = \frac{\partial f}{\partial y}$}, Amer. Math. Monthly, {\bf 82}:5 (1975) 530--531.

\bibitem{E} R.~Engelking, {\em General Topology}, Heldermann Verlag, Berlin, 1989.

\bibitem{KM} A.K.~Kalancha, V.K.~Maslyuchenko, {\em A generalization of
Bruckner-Petruska-Preiss-Thomson theorem}, Mat. Stud. {\bf
11}:1 (1999), 48--52.

\bibitem{M} V.K.~Maslyuchenko, {\em One property of partial derivatives},
- Ukr. Mat. Zh. {\bf 39}:4 (1987), 529--531.

\bibitem{MM} V.K.~Maslyuchenko, V.V.~Mykhaylyuk, {\em Solving of partial
differential equations under minimal conditions},  Zh. Mat. Fiz. Anal. Geom. {\bf 4}:2 
(2008), 252--266.

\bibitem{Mauldin} R.D.~Mauldin (Ed.), {\em The Scottish Book. Mathematcs from the Scottish Caf\'e},  Birkh\"auser, Boston, 1981.


\bibitem{Tol1} G.P.~Tolctov, {\em On partial derivatives}, Izvestiya Akad. Nauk SSSR. Ser. Mat.
 {\bf 13} (1949) 425--446.

\bibitem{Tol2} G.P.~ Tolctov, {\em On the second mixed derivative}, Mat. Sbornik. {\bf 24(66)}:1 (1949) 27--51.

\end{thebibliography}
\end{document}